\documentclass{amsart}
\usepackage[dvips,final]{graphics}
\usepackage{array}
\usepackage{arydshln}
\usepackage[makeroom]{cancel}
 \usepackage[all]{xy}
 \usepackage{url}
\usepackage{multirow, blkarray}
\usepackage{booktabs}
\usepackage{textcomp}
 \usepackage[final]{epsfig}
 \usepackage{color}
\usepackage[T1]{fontenc}      
\usepackage[english,french]{babel}
\usepackage[utf8]{inputenc}
\usepackage{blindtext}

\usepackage{amsfonts,amscd,array, mathdots, epigraph}
\usepackage{amsmath}
\usepackage{amssymb}
\usepackage{amsthm}
\usepackage{mathrsfs}
\usepackage{stmaryrd}
\usepackage{slashbox}
\usepackage{diagbox}
\usepackage{enumitem}

\usepackage{ulem}
\usepackage{tikz}
\usepackage{xcolor}
\usepackage{multicol}
\definecolor{ufogreen}{rgb}{0.24, 0.82, 0.44}

\begin{document}


\newtheorem{theorem}{Théorème}[section]
\newtheorem{theore}{Théorème}
\newtheorem{definition}[theorem]{Définition}
\newtheorem{proposition}[theorem]{Proposition}
\newtheorem{corollary}[theorem]{Corollaire}
\newtheorem*{con}{Conjecture}
\newtheorem*{remark}{Remarque}
\newtheorem*{remarks}{Remarques}
\newtheorem*{pro}{Problème}
\newtheorem*{examples}{Exemples}
\newtheorem*{example}{Exemple}
\newtheorem{lemma}[theorem]{Lemme}


\title{Solutions monomiales minimales irréductibles dans $SL_{2}(\mathbb{Z}/p^{n}\mathbb{Z})$}

\author{Flavien Mabilat}

\date{}


\address{
Flavien Mabilat,
Laboratoire de Mathématiques de Reims,
UMR9008 CNRS et Université de Reims Champagne-Ardenne, 
U.F.R. Sciences Exactes et Naturelles 
Moulin de la Housse - BP 1039 
51687 Reims cedex 2,
France
}
\email{flavien.mabilat@univ-reims.fr}

\maketitle

\selectlanguage{french}
\begin{abstract}
Dans cette article, on va s'intéresser à la combinatoire des sous-groupes de congruence du groupe modulaire. Plus précisément, on va se consacrer aux solutions monomiales minimales. Celles-ci sont les solutions d'une équation matricielle (apparaissant également lors de l'étude des frises de Coxeter), modulo un entier $N$, dont toutes les composantes sont identiques et minimales pour cette propriété. L'objectif ici est de caractériser les solutions de ce type possédant une certaine propriété d'irréductibilité lorsque $N$ est égale à la puissance d'un nombre premier. 
\\
\end{abstract}

\selectlanguage{english}
\begin{abstract}
In this article, we study the combinatorics of congruence subgroups of the modular group. More precisely, we consider the notion of minimal monomial solutions. These are the solutions of a matrix equation (also appearing in the study of Coxeter friezes), modulo an integer $N$, whose components are identical and minimal for this property. The objective here is to characterize the solutions of this type verifying a property of irreducibility modulo a prime power.
\\
\end{abstract}

\selectlanguage{french}

\thispagestyle{empty}

\noindent {\bf Mots clés:} groupe modulaire; solution monomiale; irréductibilité 
\\
\begin{flushright}
\og \textit{Comprendre c'est avant tout unifier.} \fg
\\Albert Camus, \textit{Le Mythe de Sisyphe.}
\end{flushright}

\section{Introduction}

Dans cette article, on va s'intéresser aux matrices de la forme suivante : \[M_{n}(a_{1},\ldots,a_{n})=\begin{pmatrix}
   a_{n} & -1 \\[4pt]
    1    & 0 
   \end{pmatrix}
\begin{pmatrix}
   a_{n-1} & -1 \\[4pt]
    1    & 0 
   \end{pmatrix}
   \cdots
   \begin{pmatrix}
   a_{1} & -1 \\[4pt]
    1    & 0 
    \end{pmatrix}.\]
Celles-ci jouent un rôle central dans l'étude d'un nombre importants d'objets mathématiques. On les retrouve notamment  dans l'étude des fractions continues négatives où elles sont reliées aux réduites. Elles interviennent également lors de la constructions des frises de Coxeter (voir \cite{BR}) ainsi que dans bien d'autres domaines (voir, par exemple, \cite{O} section 1.3). Un de ceux-ci, sur lequel nous allons maintenant nous concentrer, est l'étude de la combinatoire du groupe modulaire \[SL_{2}(\mathbb{Z})=
\left\{
\begin{pmatrix}
a & b \\
c & d
   \end{pmatrix}
 \;\vert\;a,b,c,d \in \mathbb{Z},\;
 ad-bc=1
\right\}.\] En effet, considérons la partie génératrice formée des deux matrices suivantes :
 \[T=\begin{pmatrix}
 1 & 1 \\[2pt]
    0    & 1 
   \end{pmatrix}, S=\begin{pmatrix}
   0 & -1 \\[2pt]
    1    & 0 
   \end{pmatrix}.
 \] 
\\On peut montrer (voir, par exemple, l'introduction de \cite{Ma3}) que pour toute matrice $A$ de $SL_{2}(\mathbb{Z})$ il existe un entier strictement positif $n$ et des entiers strictement positifs $a_{1},\ldots,a_{n}$ tels que \[A=T^{a_{n}}ST^{a_{n-1}}S\cdots T^{a_{1}}S=\begin{pmatrix}
   a_{n} & -1 \\[4pt]
    1    & 0 
   \end{pmatrix}
\begin{pmatrix}
   a_{n-1} & -1 \\[4pt]
    1    & 0 
   \end{pmatrix}
   \cdots
   \begin{pmatrix}
   a_{1} & -1 \\[4pt]
    1    & 0 
    \end{pmatrix}=M_{n}(a_{1},\ldots,a_{n}).\]
	
Malheureusement, l'écriture d'un élément de $SL_{2}(\mathbb{Z})$ sous cette forme n'est pas unique (pour une façon d'assurer l'unicité d'une écriture de cette forme on peut consulter \cite{MO} section 6).
\\
\\ \indent Ceci amène naturellement à chercher les différentes écritures d'une matrice, ou d'un ensemble de matrices, du groupe modulaire sous cette forme. On s'intéresse particulièrement au cas de la matrice $Id$. Pour cela, on considère l'équation suivante: \begin{equation}
\label{a}
\tag{$E$}
M_{n}(a_1,\ldots,a_n)=\pm Id.
\end{equation} 

\noindent On dispose des solutions de celle-ci sur différents sous-ensembles de $\mathbb{C}$, notamment $\mathbb{N}^{*}$ (voir \cite{O} Théorèmes 1 et 2 et \cite{CO2} Théorème 2.2), $\mathbb{N}$ (voir \cite{C} Théorème 3.1) et $\mathbb{Z}$ (voir \cite{C} Théorème 3.2). De plus, on a, dans ces trois cas, une description combinatoire des solutions en termes de découpages de polygones convexes. On peut également résoudre l'équation \eqref{a} sur $\mathbb{Z}[\alpha]$, avec $\alpha$ un nombre complexe transcendant (voir \cite{Ma2} Théorème 2.7). La résolution de celle-ci sur d'autres sous-ensembles de $\mathbb{C}$ est encore un problème ouvert (voir \cite{C} problème ouvert 4.1).
\\
\\ \indent On va s'intéresser ici aux cas des anneaux $\mathbb{Z}/N\mathbb{Z}$, c'est-à-dire à l'étude sur $\mathbb{Z}/N\mathbb{Z}$ de l'équation :
\begin{equation}
\label{p}
\tag{$E_{N}$}
M_{n}(a_1,\ldots,a_n)=\pm Id.
\end{equation} On dira, en particulier, qu'une solution de \eqref{p} est de taille $n$ si cette solution est un $n$-uplet d'éléments de $\mathbb{Z}/N\mathbb{Z}$. L'objectif principal de cette étude est de connaître toutes les écritures des éléments des sous-groupes de congruence suivants: \[\hat{\Gamma}(N)=\{A \in SL_{2}(\mathbb{Z})~{\rm tel~que}~A= \pm Id~( {\rm mod}~N)\}\] sous la forme $M_{n}(a_1,\ldots,a_n)$ avec les $a_{i}$ des entiers strictement positifs. 
\\
\\ \indent Un certain nombre de résultats sur l'équation \eqref{p} ont déjà été obtenus lors de précédents travaux (voir \cite{Ma1, M, M2, Ma4}). Le point central de ceux-ci est l'utilisation d'une notion de solutions irréductibles, déjà utilisée par M.\ Cuntz pour l'étude des frises de Coxeter, à partir de laquelle on peut construire l'ensemble des solutions (voir section suivante). Celle-ci a permis la résolution complète de \eqref{p} pour $N \leq 6$ (voir \cite{M} section 4) ainsi que l'obtention de plusieurs résultats généraux d'irréductibilité. 
\\
\\ \indent La majeure partie de ces résultats concernent les solutions monomiales minimales qui sont les solutions de \eqref{p} dont toutes les composantes sont identiques et minimales pour cette propriété (voir \cite{M} section 3.3 et la section suivante). Parmi ces résultats, l'un des plus marquants est l'irréductibilité de toutes les solutions monomiales minimales non nulles lorsque $N$ est premier (voir \cite{M} Théorème 3.16). Ce résultat de classification des solutions monomiales minimales irréductibles est, pour le moment, le seul dont nous disposons. Notre objectif est ici de généraliser ce théorème en classifiant les solutions monomiales minimales irréductibles lorsque $N$ est égale à la puissance d'un nombre premier. Pour cela, on va donner, dans la partie suivante, les définitions qui nous seront utiles pour la suite ainsi que le résultat principal. On démontrera ensuite celui-ci dans la section \ref{DT}. Pour terminer, on donnera, dans la section \ref{taille}, quelques résultats sur la taille de certaines solutions monomiales minimales.

\section{Définitions et résultat principal}
\label{RP}    

L'objectif de cette section est de rappeler les définitions nécessaires à l'étude de l'équation \eqref{p}, évoquées dans la section précédente, et, introduites précédemment dans \cite{C} et \cite{M}. Il s'agit également d'énoncer le résultat principal de l'étude qui va être menée dans la suite. Sauf mention contraire, $N$ désigne un entier naturel supérieur à $2$, et, s'il n'y a pas d'ambiguïté sur $N$, on note $\overline{a}=a+N\mathbb{Z}$ (avec $a \in \mathbb{Z}$). $\mathbb{P}$ représente l'ensemble des nombres premiers et $\varphi$ la fonction indicatrice d'Euler.

\begin{definition}[\cite{C}, lemme 2.7]
\label{21}

Soient $(\overline{a_{1}},\ldots,\overline{a_{n}}) \in (\mathbb{Z}/N \mathbb{Z})^{n}$ et $(\overline{b_{1}},\ldots,\overline{b_{m}}) \in (\mathbb{Z}/N \mathbb{Z})^{m}$. On définit l'opération ci-dessous: \[(\overline{a_{1}},\ldots,\overline{a_{n}}) \oplus (\overline{b_{1}},\ldots,\overline{b_{m}})= (\overline{a_{1}+b_{m}},\overline{a_{2}},\ldots,\overline{a_{n-1}},\overline{a_{n}+b_{1}},\overline{b_{2}},\ldots,\overline{b_{m-1}}).\] Le $(n+m-2)$-uplet obtenu est appelé la somme de $(\overline{a_{1}},\ldots,\overline{a_{n}})$ avec $(\overline{b_{1}},\ldots,\overline{b_{m}})$.

\end{definition}

\begin{examples}

{\rm On donne ci-dessous quelques exemples de sommes :
\begin{itemize}
\item $(\overline{3},\overline{2},\overline{1}) \oplus (\overline{1},\overline{2},\overline{3})= (\overline{6},\overline{2},\overline{2},\overline{2})$;
\item $(\overline{-2},\overline{0},\overline{-1},\overline{1}) \oplus (\overline{3},\overline{-2},\overline{2}) = (\overline{0},\overline{0},\overline{-1},\overline{4},\overline{-2})$;
\item $n \geq 2$, $(\overline{a_{1}},\ldots,\overline{a_{n}}) \oplus (\overline{0},\overline{0}) = (\overline{0},\overline{0}) \oplus (\overline{a_{1}},\ldots,\overline{a_{n}})=(\overline{a_{1}},\ldots,\overline{a_{n}})$.
\end{itemize}
}
\end{examples}

L'opération $\oplus$ définie ci-dessus n'est ni commutative ni associative (voir \cite{WZ} exemple 2.1). En revanche, celle-ci possède la propriété suivante : si $(\overline{b_{1}},\ldots,\overline{b_{m}})$ est une solution de \eqref{p} alors la somme $(\overline{a_{1}},\ldots,\overline{a_{n}}) \oplus (\overline{b_{1}},\ldots,\overline{b_{m}})$ est solution de \eqref{p} si et seulement si $(\overline{a_{1}},\ldots,\overline{a_{n}})$ est solution de \eqref{p} (voir \cite{C,WZ} et \cite{M} proposition 3.7). 

\begin{definition}[\cite{C}, définition 2.5]
\label{22}

 Soient $(\overline{a_{1}},\ldots,\overline{a_{n}}) \in (\mathbb{Z}/N \mathbb{Z})^{n}$ et $(\overline{b_{1}},\ldots,\overline{b_{n}}) \in (\mathbb{Z}/N \mathbb{Z})^{n}$. On dit que $(\overline{a_{1}},\ldots,\overline{a_{n}}) \sim (\overline{b_{1}},\ldots,\overline{b_{n}})$ si $(\overline{b_{1}},\ldots,\overline{b_{n}})$ est obtenu par permutation circulaire de $(\overline{a_{1}},\ldots,\overline{a_{n}})$ ou de $(\overline{a_{n}},\ldots,\overline{a_{1}})$.

\end{definition}

On vérifie sans difficulté que $\sim$ est une relation d'équivalence sur les $n$-uplets d'éléments de $\mathbb{Z}/N \mathbb{Z}$ (voir \cite{WZ} lemme 1.7). D'autre part, si $(\overline{a_{1}},\ldots,\overline{a_{n}}) \sim (\overline{b_{1}},\ldots,\overline{b_{n}})$ alors $(\overline{a_{1}},\ldots,\overline{a_{n}})$ est solution de \eqref{p} si et seulement si $(\overline{b_{1}},\ldots,\overline{b_{n}})$ est solution de \eqref{p} (voir \cite{C} proposition 2.6).
\\
\\On peut maintenant définir la notion d'irréductibilité annoncée dans l'introduction :

\begin{definition}[\cite{C}, définition 2.9]
\label{23}

Une solution $(\overline{c_{1}},\ldots,\overline{c_{n}})$ avec $n \geq 3$ de \eqref{p} est dite réductible s'il existe une solution de \eqref{p} $(\overline{b_{1}},\ldots,\overline{b_{l}})$ et un $m$-uplet $(\overline{a_{1}},\ldots,\overline{a_{m}})$ d'éléments de $\mathbb{Z}/N \mathbb{Z}$ tels que \begin{itemize}
\item $(\overline{c_{1}},\ldots,\overline{c_{n}}) \sim (\overline{a_{1}},\ldots,\overline{a_{m}}) \oplus (\overline{b_{1}},\ldots,\overline{b_{l}})$;
\item $m \geq 3$ et $l \geq 3$.
\end{itemize}
Une solution est dite irréductible si elle n'est pas réductible.

\end{definition}

\begin{remark} 

{\rm On ne considère pas $(\overline{0},\overline{0})$ comme une solution irréductible de \eqref{p}.}

\end{remark}

\indent Au cours des différents travaux déjà menés sur l'équation \eqref{p}, on a consacré une attention particulière à la notion de solutions monomiales rappelée dans la définition qui suit :

\begin{definition}[\cite{M}, définition 3.9]
\label{24}

i)~Soient $n \in \mathbb{N}^{*}$ et $\overline{k} \in \mathbb{Z}/N\mathbb{Z}$. On appelle solution $(n,\overline{k})$-monomiale un $n$-uplet d'éléments de $\mathbb{Z}/ N \mathbb{Z}$ constitué uniquement de $\overline{k}$ et solution de \eqref{p}. 
\\
\\ ii)~On appelle solution monomiale une solution pour laquelle il existe $m \in \mathbb{N}^{*}$ et $\overline{l} \in \mathbb{Z}/N\mathbb{Z}$ tels qu'elle est $(m,\overline{l})$-monomiale.
\\
\\ iii)~On appelle solution $\overline{k}$-monomiale minimale une solution $(n,\overline{k})$-monomiale avec $n$ le plus petit entier pour lequel il existe une solution $(n,\overline{k})$-monomiale.
\\
\\ iv)~On appelle solution monomiale minimale une solution $\overline{k}$-monomiale minimale pour un $\overline{k} \in \mathbb{Z}/N\mathbb{Z}$.

\end{definition}

Notons que pour tout $\overline{k} \in \mathbb{Z}/N\mathbb{Z}$ il y a existence et unicité de la solution $\overline{k}$-monomiale minimale de \eqref{p}, puisque $M_{1}(\overline{k})$ est d'ordre fini dans $PSL_{2}(\mathbb{Z}/N\mathbb{Z})$. Un certain nombre de propriétés d'irréductibilité pour ces solutions ont été démontrées précédemment (voir \cite{M, M2,Ma4} et la section \ref{RP}). En particulier, on a :

\begin{theorem}[\cite{M}, Théorème 3.16]
\label{25}

Si $N$ est premier alors les solutions monomiales minimales non nulles de \eqref{p} sont irréductibles. 

\end{theorem}

On souhaite ici aller plus loin en classifiant les solutions monomiales minimales irréductibles de \eqref{p} lorsque $N$ est la puissance d'un nombre premier. Notons que l'on sait déjà que toutes les solutions monomiales minimales non nulles de $(E_{4})$ et $(E_{8})$ sont irréductibles (voir \cite{Ma4} section 3.2). Dans ce texte, on va démontrer le résultat ci-dessous :

\begin{theorem}
\label{26}

Soient $p$ un nombre premier, $n$ un entier naturel non nul et $N=p^{n}$. Soit $k \in \mathbb{Z}$.
\\
\\i) Si $p \neq 2$. La solution $\overline{k}$-monomiale minimale de \eqref{p} est irréductible si et seulement si $p$ ne divise pas $k$. En particulier, \eqref{p} a $\varphi(p^{n})=p^{n-1}(p-1)$ solutions monomiales minimales irréductibles.
\\
\\ii) Si $p=2$. La solution $\overline{k}$-monomiale minimale de \eqref{p} est irréductible si et seulement si  une des propriétés suivantes est vérifiée :
\begin{itemize}
\item $k$ est impair; 
\item $\overline{k}=\overline{2^{n-1}}$;
\item $n \geq 2$ et il existe un entier impair $a$ tel que $k=2a$.
\end{itemize}

\noindent En particulier, si $n \geq 3$, \eqref{p} a $3 \times 2^{n-2}+1$ solutions monomiales minimales irréductibles.

\end{theorem}

Ce théorème, démontré dans la section suivante, permet d'obtenir une description exhaustive des solutions monomiales minimales irréductibles lorsque $N$ est la puissance d'un nombre premier. Notons que, comme dans de nombreux théorèmes de classification (notamment en théorie des groupes), le cas $p=2$ est différent des autres. Dans la dernière partie, on utilise un certain nombre d'éléments obtenus dans la preuve de ce résultat pour avoir des informations sur la taille de certaines solutions monomiales minimales.

\section{Démonstration du théorème \ref{26}}
\label{DT}

\noindent Dans cette section, on va démontrer le résultat décrit précédemment.

\subsection{Résultats préliminaires sur les solutions monomiales minimales}
\label{RP}

L'objectif de cette partie est de donner quelques résultats sur les solutions monomiales minimales qui nous seront utiles dans la suite. On commence par rappeler un certain nombre d'éléments déjà connus :

\begin{proposition}[\cite{M}, section 3.1]
\label{spe}

i)~\eqref{p} n'a pas de solution de taille 1.
\\ii)~$(\overline{0},\overline{0})$ est l'unique solution de \eqref{p} de taille 2.
\\iii)~$(\overline{1},\overline{1},\overline{1})$ et $(\overline{-1},\overline{-1},\overline{-1})$ sont les seules solutions de \eqref{p} de taille 3.
\\iv)~Les solutions de \eqref{p} de taille 4 sont de la forme $(\overline{-a},\overline{b},\overline{a},\overline{-b})$ avec $\overline{ab}=\overline{0}$ et $(\overline{a},\overline{b},\overline{a},\overline{b})$ avec $\overline{ab}=\overline{2}$. 

\end{proposition}

\begin{proposition}[\cite{M}, proposition 3.15]
\label{31}

Soient $n \in \mathbb{N}^{*}$, $n \geq 3$ et $(\overline{a},\overline{b},\overline{k}) \in (\mathbb{Z}/N\mathbb{Z})^{3}$. 
\\Si $(\overline{a},\overline{k},\overline{k},\ldots,\overline{k},\overline{b}) \in (\mathbb{Z}/N\mathbb{Z})^{n}$ est solution de \eqref{p} alors $\overline{a}=\overline{b}$ et on a \[\overline{a}(\overline{a}-\overline{k})=\overline{0}.\]
	
\end{proposition}

Le théorème \ref{25} nous fournit un premier résultat d'irréductibilité des solutions monomiales minimales. On regroupe dans l'énoncé suivant d'autres résultats d'irréductibilité qui nous seront utiles dans la suite (voir \cite{M,M2,Ma4}) :

\begin{theorem}
\label{32}

\begin{itemize}
\item la solution $\overline{k}$-monomiale minimale de \eqref{p} est irréductible si et seulement si la solution $\overline{-k}$-monomiale minimale de \eqref{p} est irréductible;
\item si $N \geq 3$, les solutions $\pm \overline{2}$-monomiales minimales de \eqref{p} sont irréductibles de taille $N$;
\item si $N$ est pair supérieur à 4, la solution $\overline{\frac{N}{2}}$-monomiale minimale de \eqref{p} est irréductible. Elle est de taille 4 si $4 \mid N$ et 6 sinon.
\\
\end{itemize}

\end{theorem}

\noindent On termine cette section par le résultat ci-dessous :

\begin{lemma}
\label{33}

Soient $N$ un entier supérieur à 2 et $\overline{k}$ un élément de $\mathbb{Z}/N\mathbb{Z}$. Soit $n$ la taille de la solution $\overline{k}$-monomiale minimale de \eqref{p}. Soit $m$ un entier naturel.
\\
\\i) Si $(\overline{a},\overline{k},\ldots,\overline{k},\overline{b}) \in (\mathbb{Z}/N\mathbb{Z})^{nm}$ est une solution de \eqref{p} alors $\overline{a}=\overline{b}=\overline{k}$.
\\
\\ii) Il n'y a pas de solution de \eqref{p} de la forme $(\overline{a},\overline{k},\ldots,\overline{k},\overline{b})$ de taille $nm+1$.
\\
\\iii) Si $(\overline{a},\overline{k},\ldots,\overline{k},\overline{b}) \in (\mathbb{Z}/N\mathbb{Z})^{nm+2}$ est une solution de \eqref{p} alors $\overline{a}=\overline{b}=\overline{0}$.

\end{lemma}

\begin{proof}

Si $m=0$ alors i) est sans objet, ii) est vrai car \eqref{p} n'a pas de solution de taille 1 et iii) est vrai car $(\overline{0},\overline{0})$ est la seule solution de \eqref{p} de taille 2. On suppose donc maintenant $m \geq 1$. Si $(\overline{a},\overline{k},\ldots,\overline{k},\overline{b}) \in (\mathbb{Z}/N\mathbb{Z})^{nm+l}$ est une solution de \eqref{p} avec $l \leq 2$. On sait déjà que $\overline{a}=\overline{b}$ (proposition \ref{31}) et $n \geq 2$. Il existe $\alpha \in \{\pm 1\}$ tel que  
\[\overline{\alpha}Id= M_{n}(\overline{k},\ldots,\overline{k}).\]
\noindent Puisque $(\overline{k},\ldots,\overline{k},\overline{a},\overline{a}) \in (\mathbb{Z}/N\mathbb{Z})^{nm+l}$ est une solution de \eqref{p}, il existe $\epsilon \in \{\pm 1\}$ tel que  
\begin{eqnarray*}
\overline{\epsilon}Id &=& M_{nm+l}(\overline{k},\ldots,\overline{k},\overline{a},\overline{a}) \\
                      &=& M_{n+l}(\overline{k},\ldots,\overline{k},\overline{a},\overline{a})M_{n(m-1)}(\overline{k},\ldots,\overline{k}) \\
											&=& M_{n+l}(\overline{k},\ldots,\overline{k},\overline{a},\overline{a})M_{n}(\overline{k},\ldots,\overline{k})^{m-1} \\
											&=& \overline{\alpha}^{m-1} M_{n+l}(\overline{k},\ldots,\overline{k},\overline{a},\overline{a}). \\
\end{eqnarray*}

\noindent Ainsi, il existe $\mu \in \{\pm 1\}$ tel que  
\[M_{n+l}(\overline{k},\ldots,\overline{k},\overline{a},\overline{a})=\overline{\mu} M_{n}(\overline{k},\ldots,\overline{k}).\]
En multipliant cette égalité $(n-2)$ fois à droite par l'inverse de $M_{1}(\overline{k})$, on obtient 
\[M_{2+l}(\overline{k},\ldots,\overline{k},\overline{a},\overline{a})=\overline{\mu} M_{2}(\overline{k},\overline{k}).\]
	
\noindent i) Si $l=0$, on a  \[M_{2}(\overline{a},\overline{a})= \begin{pmatrix}
   \overline{a^{2}-1} & -\overline{a} \\
    \overline{a}    & \overline{-1} 
   \end{pmatrix}=\overline{\mu} \begin{pmatrix}
   \overline{k^{2}-1} & -\overline{k} \\
    \overline{k}    & \overline{-1} 
   \end{pmatrix}= \overline{\mu} M_{2}(\overline{k},\overline{k}).\]
	
\noindent Donc, $\overline{\mu}=\overline{1}$ et $\overline{a}=\overline{k}$.
\\
\\ii) Si $l=1$, on a  \[M_{3}(\overline{k},\overline{a},\overline{a})= \overline{\mu} M_{2}(\overline{k},\overline{k}).\]

\noindent En multipliant à droite par l'inverse de $M_{1}(\overline{k})$, on obtient 

\[M_{2}(\overline{a},\overline{a})= \begin{pmatrix}
   \overline{a^{2}-1} & \overline{-a} \\
    \overline{a}    & \overline{-1} 
   \end{pmatrix}=\overline{\mu} \begin{pmatrix}
   \overline{k} & -\overline{1} \\
    \overline{1}    & \overline{0} 
   \end{pmatrix}= \overline{\mu} M_{1}(\overline{k}).\]
	
\noindent Ainsi, $\overline{-1}=\overline{0}$ ce qui est absurde puisque $N \geq 2$. Donc, il n'y a pas de solution de \eqref{p} de la forme $(\overline{a},\overline{k},\ldots,\overline{k},\overline{b})$ de taille $nm+1$.
\\
\\iii) Si $l=2$, on a  \[M_{4}(\overline{k},\overline{k},\overline{a},\overline{a})= \overline{\mu} M_{2}(\overline{k},\overline{k}).\] 

\noindent En multipliant 2 fois à droite par l'inverse de $M_{1}(\overline{k})$, on obtient $M_{2}(\overline{a},\overline{a})= \overline{\mu} Id$. Donc, $\overline{\mu}=\overline{-1}$ et $\overline{a}=\overline{0}$.

\end{proof}

\begin{remark}
{\rm Dans la preuve ci-dessus, on a utilisé les deux formules suivantes qui découlent directement de la définition des matrices $M_{n}(\overline{a_{1}},\ldots,\overline{a_{n}})$ :
\begin{itemize}
\item $M_{nm+l}(\overline{k},\ldots,\overline{k},\overline{a},\overline{a})=M_{n+l}(\overline{k},\ldots,\overline{k},\overline{a},\overline{a})M_{n(m-1)}(\overline{k},\ldots,\overline{k})$;
\item $M_{n(m-1)}(\overline{k},\ldots,\overline{k})=M_{n}(\overline{k},\ldots,\overline{k})^{m-1}$.
\end{itemize} 

\noindent Ces dernières sont intéressantes à noter car elles seront utiles dans la suite.
}

\end{remark}

\subsection{Cas où $k$ est premier avec $p$}
\label{pre}

On va démontrer ici que la solution $\overline{k}$-monomiale minimale de \eqref{p} est irréductible lorsque $p$ ne divise pas $k$. On commence par un résultat intermédiaire.

\begin{lemma}
\label{34}

Soient $p \in \mathbb{P}$, $n \in \mathbb{N}^{*}$ et $N=p^{n}$. Soit $k \in \mathbb{Z}$ premier avec $p$. $\overline{0}$ et $\overline{k}$ sont les seules racines de $X(X-\overline{k})$.

\end{lemma}

\begin{proof}

$\overline{0}$ et $\overline{k}$ sont racines de $X(X-\overline{k})$. Soit $\overline{a} \in \mathbb{Z}/N\mathbb{Z}$ tel que $\overline{a}(\overline{a}-\overline{k})=\overline{0}$.
\\
\\Si $\overline{a}$ est un élément inversible de $\mathbb{Z}/N\mathbb{Z}$. On a \[\overline{a}(\overline{a}-\overline{k})=\overline{0} \Longleftrightarrow (\overline{a}-\overline{k})=\overline{0} \Longleftrightarrow \overline{a}=\overline{k}.\] 

\noindent Si $\overline{a}$ n'est pas un élément inversible de $\mathbb{Z}/N\mathbb{Z}$. Il existe un entier $m$ dans $[\![1;n-1]\!]$ et un entier $i$ dans $[\![0;p^{n-m}]\!]$ tels que $\overline{a}=\overline{ip^{m}}$. Si $i=0$ alors $\overline{a}=\overline{0}$. On suppose maintenant $i \neq 0$ et $p$ ne divise pas $i$.
\\
\\On a \[\overline{a}(\overline{a}-\overline{k})=\overline{ip^{m}}(\overline{ip^{m}}-\overline{k})=\overline{0}.\] Ainsi, $p^{n}$ divise $ip^{m}(ip^{m}-k)$ et donc $p^{n-m}$ divise $i(ip^{m}-k)$. Comme $p^{n-m}$ et $i$ sont premiers entre eux, on a, par le lemme de Gauss, $p^{n-m}$ divise $ip^{m}-k$. En particulier, $p$ divise $ip^{m}-k$. Donc, $p$ divise $k$ ce qui est absurde.
\\
\\Donc, $\overline{a} \in \{\overline{0},\overline{k}\}$.
\\
\\Ainsi, $\overline{0}$ et $\overline{k}$ sont les seules racines de $X(X-\overline{k})$.

\end{proof}

\noindent Cela nous permet d'avoir le résultat souhaité :

\begin{proposition}
\label{35}

Soient $p \in \mathbb{P}$, $n \in \mathbb{N}^{*}$ et $N=p^{n}$. Si $\overline{k}$ est inversible dans $\mathbb{Z}/N\mathbb{Z}$ alors la solution $\overline{k}$-monomiale minimale de \eqref{p} est irréductible.

\end{proposition}

\begin{proof}

Soit $\overline{k}$ un élément inversible de $\mathbb{Z}/N\mathbb{Z}$. Soit $m \in \mathbb{N}^{*}$ tel que $(\overline{k},\ldots,\overline{k}) \in (\mathbb{Z}/N\mathbb{Z})^{m}$ soit monomiale minimale. On suppose par l'absurde que cette solution peut s'écrire comme une somme de deux solutions non triviales.
\\
\\Il existe $(\overline{a_{1}},\ldots,\overline{a_{l}})$ et $(\overline{b_{1}},\ldots,\overline{b_{l'}})$ solutions de \eqref{p} différentes de $(\overline{0},\overline{0})$ avec $l+l'=m+2$ et $l,l' \geq 3$ telles que \[(\overline{k},\ldots,\overline{k})=(\overline{b_{1}+a_{l}},\overline{b_{2}},\ldots,\overline{b_{l'-1}},\overline{b_{l'}+a_{1}},\overline{a_{2}},\ldots,\overline{a_{l-1}}).\] On a donc $\overline{a_{2}}=\ldots=\overline{a_{l-1}}=\overline{k}$. Comme $(\overline{a_{1}},\ldots,\overline{a_{l}})$ est solution de \eqref{p}, on a, par la proposition \ref{31}, $\overline{a_{1}}=\overline{a_{l}}=\overline{a}$ avec $\overline{a}(\overline{a}-\overline{k})=\overline{0}$. Par le lemme précédent, on a $\overline{a}=\overline{0}$ ou $\overline{a}=\overline{k}$. 
\\
\\Si $\overline{a}=\overline{0}$ alors \[(\overline{0},\overline{a_{2}},\ldots,\overline{a_{l-1}},\overline{0}) \sim (\overline{a_{2}},\ldots,\overline{a_{l-1}}) \oplus (\overline{0},\overline{0},\overline{0},\overline{0}).\]  Comme $(\overline{0},\overline{a_{2}},\ldots,\overline{a_{l-1}},\overline{0})$ et $(\overline{0},\overline{0},\overline{0},\overline{0})$ sont solutions de \eqref{p}, $(\overline{a_{2}},\ldots,\overline{a_{l-1}})=(\overline{k},\ldots,\overline{k}) \in (\mathbb{Z}/N\mathbb{Z})^{l-2}$ est solution de \eqref{p}, ce qui contredit la minimalité de la solution. 
\\
\\Ainsi, $\overline{a}=\overline{k}$. Par minimalité de la solution, on a $l \geq m$ ce qui implique $l' \leq 2$. Donc, $l'=2$ et $(\overline{b_{1}},\ldots,\overline{b_{l'}})=(\overline{0},\overline{0})$ ce qui est absurde.

\end{proof}

\begin{remarks}

{\rm i) Si $N=p \in \mathbb{P}$ alors tous les éléments non nuls de $\mathbb{Z}/N\mathbb{Z}$ sont inversibles. Donc, la proposition \ref{35} implique le théorème \ref{25}.
\\
\\ii) Ce résultat combiné avec le théorème \ref{32} nous permet de retrouver que toutes les solutions monomiales minimales non nulles de $(E_{4})$ et $(E_{8})$ sont irréductibles.
\\
\\iii) La solution $\overline{k}$-monomiale minimale de $(E_{p^{n}})$ est de taille inférieure à $3p^{n}$ (voir \cite{M2} Théorème 3.3).
\\
\\iv) Si $N$ a plus d'un diviseur premier alors le résultat n'est, en général, plus valide. Par exemple, prenons $N=10$ et $k=3$. $k$ et $N$ sont premiers entre eux mais la solution $\overline{3}$-monomiale minimale de $(E_{10})$ est réductible. En effet, celle-ci est de taille 15 et on peut l'écrire comme une somme à l'aide de la solution $(\overline{8},\overline{3},\overline{3},\overline{3},\overline{8})$.
}

\end{remarks}

\subsection{Cas où $N$ est la puissance d'un nombre premier impair}
\label{imp}

On souhaite ici terminer la preuve du premier point du théorème \ref{26}. Pour effectuer cela, il reste à considérer la situation où $k=ap^{m}$ avec $a$ premier avec $p$. Avant de s'intéresser à celle-ci, on donne quelques résultats préliminaires :

\begin{lemma}
\label{36}

i) (\cite{M}, lemme 3.13]) Soient $(n,l) \in \mathbb{N}^{2}$, $n \geq 2$, $l \geq 2$ et $j \in [\![1;n-1]\!]$. $l^{n-j}$ divise $\binom{l^{n-1}}{j}$.
\\
\\ii) (\cite{M2},lemme 3.11) Soient $(n,l) \in \mathbb{N}^{2}$, $n \geq 3$, $l \geq 2$ et $j \in [\![2;n-1]\!]$. $l^{n-j}$ divise $\binom{2l^{n-2}}{j}$.

\end{lemma}

\begin{lemma}
\label{37}

Soient $x \in \mathbb{R}$, $x \geq 2$ et $n \in \mathbb{N}^{*}$. On a
\\i) $x^{n} >n$;
\\ii) $2x^{n-1} >n$.

\end{lemma}

\begin{proof}

i) On procède par récurrence sur $n$. $x^{1}=x \geq 2 >1$. Supposons qu'il existe un $n \in \mathbb{N}^{*}$ tel que $x^{n} >n$. $x^{n+1}=xx^{n}>xn \geq 2n \geq n+1$. Le résultat est démontré par récurrence.
\\
\\ii) On procède par récurrence sur $n$. $2x^{0}=2 >1$. Supposons qu'il existe un $n \in \mathbb{N}^{*}$ tel que $2x^{n-1} >n$. 
\\$2x^{n}=x(2x^{n-1})>xn \geq 2n \geq n+1$. Le résultat est démontré par récurrence.

\end{proof}

\begin{lemma}
\label{38}

$\forall (j,m) \in (\mathbb{N}^{*})^{2}$, $jm-j-m+1 \geq 0$.

\end{lemma}

\begin{proof}

On procède par récurrence sur $j$.
\\
\\Si $j=1$ alors, pour tout $m \in \mathbb{N}^{*}$, $jm-j-m+1=0$. 
\\
\\Supposons qu'il existe un $j \in \mathbb{N}^{*}$ tel que, pour tout $m \in \mathbb{N}^{*}$, $jm-j-m+1 \geq 0$. On a pour tout $m \in \mathbb{N}^{*}$ :
\[(j+1)m-(j+1)-m+1=jm+m-j-1-m+1=(jm-j-m+1)+m-1 \geq m-1 \geq 0.\]
Le résultat est démontré par récurrence.

\end{proof}

\noindent On retourne maintenant à notre problème initial en démontrant le résultat qui suit.

\begin{lemma}
\label{39}

Soient $N=l^{n}$, $l,n \geq 2$, $1 \leq m \leq n-1$ et $a$ un entier alors $(\overline{al^{m}},\ldots,\overline{al^{m}}) \in (\mathbb{Z}/N\mathbb{Z})^{2l^{n-m}}$ est solution de \eqref{p}.

\end{lemma}

\begin{proof}

Posons $A=\begin{pmatrix}
   a^{2}l^{m} & -a \\
   a     & 0
\end{pmatrix}$. On a 
\begin{eqnarray*}
M_{2l^{n-m}}(\overline{al^{m}},\ldots,\overline{al^{m}}) &=& \left(\begin{pmatrix}
   \overline{al^{m}}   & \overline{-1} \\
   \overline{1} & \overline{0}
\end{pmatrix}\begin{pmatrix}
   \overline{al^{m}}   & \overline{-1} \\
   \overline{1} & \overline{0}
\end{pmatrix}\right)^{l^{n-m}} \\
                                         &=& \begin{pmatrix}
   \overline{a^{2}l^{2m}-1}   & \overline{-al^{m}} \\
   \overline{al^{m}} & \overline{-1}
\end{pmatrix}^{l^{n-m}} \\
                                         &=& \overline{\left(-Id+l^{m}\begin{pmatrix}
   a^{2}l^{m} & -a \\
   a     & 0
\end{pmatrix}\right)^{l^{n-m}}} \\
																				 &=& \overline{ \sum_{j=0}^{l^{n-m}} \binom{l^{n-m}}{j} l^{jm}(-1)^{-j+l^{n-m}}A^{j}}~({\rm bin\hat{o}me~de~Newton}) \\
																				 &=& \overline{ (-1)^{l^{n-m}}\binom{l^{n-m}}{0}Id + \sum_{j=1}^{n-m} \binom{l^{(n-m+1)-1}}{j} l^{jm}(-1)^{-j+l^{n-m}} A^{j}} \\
																				 &+& \overline{\sum_{j=n-m+1}^{l^{n-m}} \binom{l^{n-m}}{j} l^{jm}(-1)^{-j+l^{n-m}} A^{j}}~(l^{n-m} >n-m, {\rm lemme}~\ref{37}~i)).  \\																				
\end{eqnarray*}

\noindent Par le lemme \ref{36} i), $l^{n-m-j+1}$ divise $\binom{l^{(n-m+1)-1}}{j}$ pour $1 \leq j \leq n-m$. Donc, $l^{n-m-j+jm+1}$ divise $\binom{l^{(n-m+1)-1}}{j} l^{jm}$. Or, par le lemme \ref{38}, $jm-j-m+1 \geq 0$ et donc $l^{n-m-j+jm+1}=l^{n}l^{jm-j-m+1}$. Ainsi, \[\overline{\sum_{j=1}^{n-m} \binom{l^{(n-m+1)-1}}{j} l^{jm}(-1)^{-j+l^{n-m}} A^{j}}=\overline{0}.\]

\noindent Si $j \geq n-m+1$ alors on a \[jm \geq nm-m^{2}+m=n+n(m-1)+m(1-m)=n+(m-1)(n-m) \geq n.\] On en déduit \[\overline{\sum_{j=n-m+1}^{l^{n-m}} \binom{l^{n-m}}{j} l^{jm}(-1)^{-j+l^{n-m}} A^{j}}=\overline{0}.\]
\\
\\Donc, $M_{2l^{n-m}}(\overline{al^{m}},\ldots,\overline{al^{m}})=\overline{ (-1)^{l^{n-m}}}Id$.

\end{proof}

\noindent On s'intéresse maintenant à l'irréductibilité de ces solutions.

\begin{lemma}
\label{310}

Soient $N=l^{n}$ avec $l>2$, $n \geq 3$, $1 \leq m \leq n-2$ et $a$ un entier. 
\\$(\overline{2al^{n-1}},\overline{al^{m}},\ldots,\overline{al^{m}},\overline{2al^{n-1}}) \in (\mathbb{Z}/N\mathbb{Z})^{2l^{n-m}-4l^{n-m-1}+2}$ est une solution de \eqref{p}.

\end{lemma}

\begin{proof}

\noindent $(\overline{2al^{n-1}},\overline{al^{m}},\ldots,\overline{al^{m}},\overline{2al^{n-1}}) \sim (\overline{2al^{n-1}},\overline{2al^{n-1}},\overline{al^{m}},\ldots,\overline{al^{m}})$. 
\\
\\Donc, $(\overline{2al^{n-1}},\overline{al^{m}},\ldots,\overline{al^{m}},\overline{2al^{n-1}})$ est solution de \eqref{p} si et seulement si $(\overline{2al^{n-1}},\overline{2al^{n-1}},\overline{al^{m}},\ldots,\overline{al^{m}})$ l'est aussi. Notons que \[2l^{n-m}-4l^{n-m-1}+2=2l^{n-m-1}(l-2)+2 \geq 2l^{n-m-1}+2 \geq 2l+2 \geq 8.\] On a 

\begin{eqnarray*}
M &=& M_{2l^{n-m}-4l^{n-m-1}+2}(\overline{2al^{n-1}},\overline{2al^{n-1}},\overline{al^{m}},\ldots,\overline{al^{m}}) \\
  &=& M_{2l^{n-m}-4l^{n-m-1}}(\overline{al^{m}},\ldots,\overline{al^{m}})M_{2}(\overline{2al^{n-1}},\overline{2al^{n-1}}).
\end{eqnarray*}

\noindent On pose $B=\begin{pmatrix}
   0   & a \\
   -a & a^{2}l^{m}
\end{pmatrix}$. On a aussi

\begin{eqnarray*}
M_{2l^{n-m}-4l^{n-m-1}}(\overline{al^{m}},\ldots,\overline{al^{m}}) &=& M_{2l^{n-m}}(\overline{al^{m}},\ldots,\overline{al^{m}})M_{4l^{n-m-1}}(\overline{al^{m}},\ldots,\overline{al^{m}})^{-1} \\ 
                                                        &=& \overline{(-1)^{l^{n-m}}} M_{4l^{n-m-1}}(\overline{al^{m}},\ldots,\overline{al^{m}})^{-1}~\\
                                                        &=& \overline{(-1)^{l^{n-m}}}\begin{pmatrix}
   \overline{0}   & \overline{1} \\
   \overline{-1} & \overline{al^{m}}
\end{pmatrix}^{4l^{n-m-1}}\\
                                                        &=& \overline{(-1)^{l^{n-m}}}(\begin{pmatrix}
   \overline{0}   & \overline{1} \\
   \overline{-1} & \overline{al^{m}}
\end{pmatrix}^{2})^{2l^{n-m-1}} \\
                                                        &=& \overline{(-1)^{l^{n-m}}}\begin{pmatrix}
   \overline{-1}   & \overline{al^{m}} \\
   \overline{-al^{m}} & \overline{-1+a^{2}l^{2m}}
\end{pmatrix}^{2l^{n-m-1}} \\
                                                        &=& \overline{(-1)^{l^{n-m}}(-Id+l^{m}\begin{pmatrix}
   0   & a \\
   -a & a^{2}l^{m}
\end{pmatrix})^{2l^{n-m-1}}} \\
																				                &=& \overline{ \sum_{j=0}^{2l^{n-m-1}} \binom{2l^{n-m-1}}{j} (-1)^{-j+l^{n-m-1}(l+2)}l^{jm}B^{j}}~{\rm (bin\hat{o}me~de~Newton)} \\
																				                &=& \overline{(-1)^{l^{n-m-1}(l+2)}Id+(-1)^{-1+l^{n-m-1}(l+2)}(2l^{n-m-1})l^{m}B} \\
                                                        &+& \overline{\sum_{j=2}^{n-m} \binom{2l^{n-m-1}}{j} (-1)^{-j+l^{n-m-1}(l+2)}l^{jm}B^{j}} \\
                                                        &+& \overline{ \sum_{j=n-m+1}^{2l^{n-m-1}} \binom{2l^{n-m-1}}{j} (-1)^{-j+l^{n-m-1}(l+2)}l^{jm}B^{j}}\\
																												& & (2l^{n-m-1}>n-m,~{\rm lemme}~\ref{37}~ii)). \\
\end{eqnarray*}

\noindent Par le lemme \ref{36} ii), $l^{n-m+1-j}$ divise $\binom{2l^{(n-m+1)-2}}{j}$ pour $2 \leq j \leq n-m$. Donc, $l^{n-m-j+jm+1}$ divise $\binom{2l^{n-m-1}}{j} l^{jm}$. Or, par le lemme \ref{38}, $jm-j-m+1 \geq 0$ et donc $l^{n-m-j+jm+1}=l^{n}l^{jm-j-m+1}$. Ainsi, \[\overline{ \sum_{j=2}^{n-m} \binom{2l^{n-m-1}}{j} (-1)^{-j+l^{n-m-1}(l+2)}l^{jm}B^{j}}=\overline{0}.\] 
\\
\\Si $j \geq n-m+1$ alors on a \[jm \geq nm-m^{2}+m=n+n(m-1)+m(1-m)=n+(m-1)(n-m) \geq n.\] On en déduit \[\overline{ \sum_{j=n-m+1}^{2l^{n-m-1}} \binom{2l^{n-m-1}}{j} (-1)^{l^{n-m-1}(l+2)-j}l^{jm}B^{j}}=\overline{0}.\]
\\
\\Ainsi, $M_{2l^{n-m}-4l^{n-m-1}}(\overline{al^{m}},\ldots,\overline{al^{m}})=\overline{(-1)^{l^{n-m-1}(l+2)}\left(Id-2l^{n-1}\begin{pmatrix}
   0   & a \\
   -a & a^{2}l^{m}
\end{pmatrix}\right)}$.
\\
\\ \noindent De plus, $M_{2}(\overline{2al^{n-1}},\overline{2al^{n-1}})=\begin{pmatrix}
   \overline{-1}   & \overline{-2al^{n-1}} \\
   \overline{2al^{n-1}} & \overline{-1}
\end{pmatrix}.$
\\
\\Donc, on a
\begin{eqnarray*}
M &=& M_{2l^{n-m}-4l^{n-m-1}}(\overline{al^{m}},\ldots,\overline{al^{m}})M_{2}(\overline{2al^{n-1}},\overline{2al^{n-1}})\\
  &=& \overline{(-1)^{l^{n-m-1}(l+2)}\left(Id-2l^{n-1}\begin{pmatrix}
   0   & a \\
   -a & a^{2}l^{m}
\end{pmatrix}\right)}\begin{pmatrix}
   \overline{-1}   & \overline{-2al^{n-1}} \\
   \overline{2al^{n-1}} & \overline{-1}
\end{pmatrix} \\
 &=& \overline{(-1)^{l^{n-m-1}(l+2)}}\left(\begin{pmatrix}
   \overline{-1}   & \overline{-2al^{n-1}} \\
   \overline{2al^{n-1}} & \overline{-1}
\end{pmatrix}-\overline{2l^{n-1}}\begin{pmatrix}
   \overline{2a^{2}l^{n-1}}   & \overline{-a} \\
   \overline{a+2a^{3}l^{n+m-1}} & \overline{2a^{2}l^{n-1}-a^{2}l^{m}}
\end{pmatrix}\right)\\
 &=& \overline{(-1)^{l^{n-m-1}(l+2)}}\left(\begin{pmatrix}
   \overline{-1}   & \overline{-2al^{n-1}} \\
   \overline{2al^{n-1}} & \overline{-1}
\end{pmatrix}-\begin{pmatrix}
   \overline{0}   & \overline{-2al^{n-1}} \\
   \overline{2al^{n-1}} & \overline{0}
\end{pmatrix}\right) \\
 &=& \overline{(-1)^{1+l^{n-m-1}(l+2)}Id}.
\end{eqnarray*}

\end{proof}

\noindent On peut maintenant démontrer le résultat principal de réductibilité.

\begin{proposition}
\label{311}

Soient $p \in \mathbb{P}$ impair, $n \geq 2$, $N=p^{n}$, $1 \leq m \leq n$ et $a$ un entier premier avec $p$. La solution $\overline{ap^{m}}$-monomiale minimale de \eqref{p} est réductible de taille $2p^{n-m}$.

\end{proposition}

\begin{proof}

Si $m=n$ alors $\overline{ap^{m}}=\overline{0}$ et la	solution est réductible de taille $2=2p^{n-n}$. On suppose donc maintenant $m \leq n-1$.
\\
\\Par le lemme \ref{39}, $(\overline{ap^{m}},\ldots,\overline{ap^{m}}) \in (\mathbb{Z}/N\mathbb{Z})^{2p^{n-m}}$ est solution de \eqref{p}. Notons $h$ la taille de la solution $\overline{ap^{m}}$-monomiale minimale de \eqref{p}.
\\
\\Comme $p^{m}$ divise $p^{n}$, le $h$-uplet ne contenant que $ap^{m}$ est une solution modulo $p^{m}$, c'est-à-dire que l'on a une solution de $(E_{p^{m}})$ ne contenant que des zéros de taille $h$. Une solution ne contenant que des zéros est toujours de taille paire. Donc, $h$ est pair. 
\\
\\Supposons, pour commencer, $n=2$ ou $m=n-1$. Par le lemme \ref{39}, $h$ divise $2p$, c'est-à-dire que l'on a $h \in \{1, 2, p, 2p\}$. Puisque $h$ est pair et $p$ impair, on a $h=2$ ou $h=2p$.  Comme $\overline{ap^{m}} \neq \overline{0}$ (car $a$ et $p$ premiers entre eux), on a, par la proposition \ref{spe}, $h=2p$. De plus, $\overline{ap^{m}}^{2}=\overline{0}$ (car $2m \geq n$). Donc, $(\overline{-ap^{m}},\overline{ap^{m}},\overline{ap^{m}},\overline{-ap^{m}})$ est solution de \eqref{p} (voir proposition \ref{spe}). Ainsi, 
\[(\overline{ap^{m}},\ldots,\overline{ap^{m}})= (\overline{2ap^{m}},\overline{ap^{m}},\ldots,\overline{ap^{m}},\overline{2ap^{m}}) \oplus (\overline{-ap^{m}},\overline{ap^{m}},\overline{ap^{m}},\overline{-ap^{m}}).\]
\noindent Comme $h \geq 5$, la solution $(\overline{2ap^{m}},\overline{ap^{m}},\ldots,\overline{ap^{m}},\overline{2ap^{m}})$ est de taille supérieure à 3. Ainsi, la solution $\overline{ap^{m}}$-monomiale minimale de \eqref{p} est réductible de taille $2p$.
\\
\\On suppose maintenant $n \geq 3$ et $m \leq n-2$. On a \[(\overline{ap^{m}},\ldots,\overline{ap^{m}})=(\overline{ap^{m}-2ap^{n-1}},\overline{ap^{m}},\ldots,\overline{ap^{m}},\overline{ap^{m}-2ap^{n-1}}) \oplus (\overline{2ap^{n-1}},\overline{ap^{m}},\ldots,\overline{ap^{m}},\overline{2ap^{n-1}}),\] avec \[(\overline{2ap^{n-1}},\overline{ap^{m}},\ldots,\overline{ap^{m}},\overline{2ap^{n-1}}) \in (\mathbb{Z}/N\mathbb{Z})^{2p^{n-m}-4p^{n-m-1}+2}\] et \[(\overline{ap^{m}-2ap^{n-1}},\overline{ap^{m}},\ldots,\overline{ap^{m}},\overline{ap^{m}-2ap^{n-1}}) \in (\mathbb{Z}/N\mathbb{Z})^{4p^{n-m-1}}.\] 
\\
\\Or, par le lemme précédent, $(\overline{2ap^{n-1}},\overline{ap^{m}},\ldots,\overline{ap^{m}},\overline{2ap^{n-1}}) \in (\mathbb{Z}/N\mathbb{Z})^{2p^{n-m}-4p^{n-m-1}+2}$ est une solution de \eqref{p}. De plus celle-ci est de taille supérieure à 3.
\\
\\$(\overline{ap^{m}-2ap^{n-1}},\overline{ap^{m}},\ldots,\overline{ap^{m}},\overline{ap^{m}-2ap^{n-1}}) \in (\mathbb{Z}/N\mathbb{Z})^{4p^{n-m-1}}$ est de taille supérieure à 4. 
\\
\\Donc, $(\overline{ap^{m}},\ldots,\overline{ap^{m}}) \in (\mathbb{Z}/N\mathbb{Z})^{2p^{n-m}}$ est une solution réductible de \eqref{p}.
\\
\\De plus, $h$ divise $2p^{n-m}$, c'est-à-dire $h$ est de la forme $h=2^{u}p^{v}$ (puisque $p$ premier). On sait également que $(\overline{ap^{m}-2ap^{n-1}},\overline{ap^{m}},\ldots,\overline{ap^{m}},\overline{ap^{m}-2ap^{n-1}}) \in (\mathbb{Z}/N\mathbb{Z})^{4p^{n-m-1}}$ est une solution de \eqref{p}. Comme $p \neq 2$ et $p \nmid a$, $\overline{ap^{m}-2ap^{n-1}} \neq \overline{ap^{m}}$. Ainsi, par le lemme \ref{33} i), $h$ ne divise pas $4p^{n-m-1}$.
\\
\\Comme $h$ est pair, il existe $i \in \mathbb{N}$, $i \leq n-m$ tel que $h=2p^{i}$. Si $i \leq n-m-1$ alors $h$ divise $4p^{n-m-1}$, ce qui n'est pas le cas. Donc, $i=n-m$ et $h=2p^{n-m}$.
\\
\\Ainsi, la solution $\overline{ap^{m}}$-monomiale minimale de \eqref{p} est réductible de taille $2p^{n-m}$.

\end{proof}

\noindent En combinant les propositions \ref{35} et \ref{311}, on obtient le premier point du théorème \ref{26}.
 
\subsection{Cas où $N$ est une puissance de deux}
\label{pair}

On va maintenant finir la preuve du théorème \ref{26} en s'intéressant au cas où $N$ est une puissance de 2. 
\\
\\ \indent Pour cela on va commencer par étudier les situations où $k=a2^{m}$ avec $n \geq 4$, $2 \leq m \leq n-2$ et $a$ impair. Le résultat du lemme \ref{39} est valide pour $l=2$ mais on ne peut malheureusement pas prolonger tel quel le lemme \ref{310}. En effet, pour $p=2$ on a $2p^{n-m}-4p^{n-m-1}+2=2$. On va donc essayer de trouver une autre solution nous permettant d'effectuer une réduction.
\\
\\Comme dans le cas précédent, on commence par donner quelques résultats préliminaires :

\begin{lemma}
\label{312}

i) $\forall (j,m) \in \mathbb{N}^{2}$, $j,m \geq 2$, $jm-j-m \geq 0$.
\\ii) Soient $n \geq 4$, $2 \leq m \leq n-2$ et $j \geq n-m$. On a $jm \geq n$.
\\iii) Si $n \geq 3$, $2^{n-1} \geq n+1$.
\\iv) Soit $n \in \mathbb{N}^{*}$. On a $2^{n-1} \geq n$.

\end{lemma}

\begin{proof}

i) On procède par récurrence sur $j$. 
\\
\\Si $j=2$, alors, pour tout $m \geq 2$, on a $jm-j-m=2m-2-m=m-2 \geq 0$. Supposons qu'il existe un $j \in \mathbb{N}$, $j \geq 2$, tel que, pour tout $m \geq 2$, $jm-j-m \geq 0$. On a, pour tout $m \geq 2$, 
\[(j+1)m-(j+1)-m=jm+m-j-1-m=(jm-j-m)+m-1 \geq m-1 \geq 0.\]
Le résultat est démontré par récurrence.
\\
\\ii) Si $j \geq n-m$ alors on a 
\begin{eqnarray*}
jm & \geq & nm-m^{2} \\
   &=& n+n(m-1)-m^{2}+1-1 \\
	 &=& n+n(m-1)-(m-1)(m+1)-1 \\
	 &=& n+(m-1)(n-m-1)-1 \\
	 & \geq & n~({\rm car}~(m-1)\geq 1~{\rm et}~(n-m-1)\geq 1).\\
\end{eqnarray*}

\noindent iii) Par récurrence sur $n$. $2^{3-1}=2^{2}=4$. Supposons qu'il existe un $n \geq 3$ tel que $2^{n-1} \geq n+1$. $2^{(n+1)-1}=2^{n}=2 \times 2^{n-1} \geq 2n+2 \geq n+2$. Le résultat est démontré par récurrence.
\\
\\ \noindent iv) Si $n \geq 3$ alors le résultat est vrai par iii). Si $n=1$ on a $2^{1-1}=2^{0}=1$. Si $n=2$ on a $2^{2-1}=2^{1}=2$. Ainsi, le résultat est démontré.

\end{proof} 

\begin{lemma}
\label{313}

Soient $N=2^{n}$ avec $n \geq 4$, $2 \leq m \leq n-2$ et $a$ un entier. $(\overline{a2^{n-1}},\overline{a2^{m}},\ldots,\overline{a2^{m}},\overline{a2^{n-1}}) \in (\mathbb{Z}/N\mathbb{Z})^{2^{n-m}+2}$ est une solution de \eqref{p}.

\end{lemma}

\begin{proof}

\noindent $(\overline{a2^{n-1}},\overline{a2^{m}},\ldots,\overline{a2^{m}},\overline{a2^{n-1}}) \sim (\overline{a2^{n-1}},\overline{a2^{n-1}},\overline{a2^{m}},\ldots,\overline{a2^{m}})$. 
\\
\\ \noindent $(\overline{a2^{n-1}},\overline{a2^{m}},\ldots,\overline{a2^{m}},\overline{a2^{n-1}})$ est solution de \eqref{p} si et seulement si $(\overline{a2^{n-1}},\overline{a2^{n-1}},\overline{a2^{m}},\ldots,\overline{a2^{m}})$ l'est aussi. On a 

\begin{eqnarray*}
M &=& M_{2^{n-m}+2}(\overline{a2^{n-1}},\overline{a2^{n-1}},\overline{a2^{m}},\ldots,\overline{a2^{m}}) \\
  &=& M_{2^{n-m}}(\overline{a2^{m}},\ldots,\overline{a2^{m}})M_{2}(\overline{a2^{n-1}},\overline{a2^{n-1}}).
\end{eqnarray*}

\noindent On pose $C=\begin{pmatrix}
   a^{2}2^{m} & -a \\
   a     & 0
\end{pmatrix}$. On a 

\begin{eqnarray*}
M_{2^{n-m}}(\overline{a2^{m}},\ldots,\overline{a2^{m}})  &=& \begin{pmatrix}
   \overline{a2^{m}}   & \overline{-1} \\
   \overline{1} & \overline{0}
\end{pmatrix}^{2^{n-m}}\\
                                                        &=& \left(\begin{pmatrix}
   \overline{a2^{m}}   & \overline{-1} \\
   \overline{1} & \overline{0}
\end{pmatrix}^{2}\right)^{2^{n-m-1}} \\
                                                        &=& \begin{pmatrix}
   \overline{a^{2}2^{2m}-1}   & \overline{-a2^{m}} \\
   \overline{a2^{m}} & \overline{-1}
\end{pmatrix}^{2^{n-m-1}} \\
                                                        &=& \overline{\left(-Id+2^{m}\begin{pmatrix}
   a^{2}2^{m} & -a \\
   a     & 0
\end{pmatrix}\right)^{2^{n-m-1}}} \\
\end{eqnarray*}

\begin{eqnarray*}
M_{2^{n-m}}(\overline{a2^{m}},\ldots,\overline{a2^{m}}) &=& \overline{ \sum_{j=0}^{2^{n-m-1}} \binom{2^{n-m-1}}{j} (-1)^{-j+2^{n-m-1}}2^{jm}C^{j}}~{\rm (bin\hat{o}me~de~Newton)}. \\
                                                        &=& \overline{(-1)^{2^{n-m-1}}Id+(-1)^{-1+2^{n-m-1}}(2^{n-m-1})2^{m}C} \\
                                                        &+& \overline{\sum_{j=2}^{n-m-1} \binom{2^{n-m-1}}{j} (-1)^{-j+2^{n-m-1}}2^{jm}C^{j}} \\
                                                        &+& \overline{ \sum_{j=n-m}^{2^{n-m-1}} \binom{2^{n-m-1}}{j} (-1)^{-j+2^{n-m-1}}2^{jm}C^{j}}~(2^{n-m-1}\geq n-m,~{\rm lemme}~\ref{312}~iv)).
\end{eqnarray*}

\noindent Si $m=n-2$ alors $n-m-1=1$ et alors 

\begin{eqnarray*}
M_{2^{n-m}}(\overline{a2^{m}},\ldots,\overline{a2^{m}})	  &=& \overline{(-1)^{2^{n-m-1}}Id+(-1)^{-1+2^{n-m-1}}(2^{n-m-1})2^{m}C} \\
                                                        &+& \overline{ \sum_{j=n-m}^{2^{n-m-1}} \binom{2^{n-m-1}}{j} (-1)^{-j+2^{n-m-1}}2^{jm}C^{j}}.
\end{eqnarray*}

\noindent Par le lemme \ref{36} i), $2^{n-m-j}$ divise $\binom{2^{n-m-1}}{j}$ pour $2 \leq j \leq n-m-1$. Donc, $2^{n-m-j+jm}$ divise $\binom{2^{n-m-1}}{j} 2^{jm}$. Or, par le lemme \ref{312} i), $jm-j-m \geq 0$ (puisque $j,m \geq 2$) et donc $2^{n-m-j+jm}=2^{n}2^{jm-j-m}$. Ainsi, \[\overline{ \sum_{j=2}^{n-m-1} \binom{2^{n-m-1}}{j} (-1)^{-j+2^{n-m-1}}2^{jm}C^{j}}=\overline{0}.\] 
\\
\\ Si $j \geq n-m$, on a, par le lemme \ref{312} ii), $jm \geq n$. Donc, $\overline{ \sum_{j=n-m}^{2^{n-m-1}} \binom{2^{n-m-1}}{j} (-1)^{2^{n-m-1}-j}2^{jm}C^{j}}=\overline{0}$.
\\
\\Ainsi, $M_{2^{n-m}}(\overline{a2^{m}},\ldots,\overline{a2^{m}})=\overline{(-1)^{2^{n-m-1}}Id+(-1)^{-1+2^{n-m-1}}(2^{n-1})\begin{pmatrix}
   a^{2}2^{m} & -a \\
   a     & 0
\end{pmatrix}}.$
\\
\\ \noindent De plus, $M_{2}(\overline{a2^{n-1}},\overline{a2^{n-1}})=\begin{pmatrix}
   \overline{-1}   & \overline{-a2^{n-1}} \\
   \overline{a2^{n-1}} & \overline{-1}
\end{pmatrix}.$
\\
\\Donc, on a
\begin{eqnarray*}
M &=& M_{2^{n-m}}(\overline{a2^{m}},\ldots,\overline{a2^{m}})M_{2}(\overline{a2^{n-1}},\overline{a2^{n-1}})\\
  &=& \overline{\left((-1)^{2^{n-m-1}}Id+(-1)^{-1+2^{n-m-1}}(2^{n-1})\begin{pmatrix}
   a^{2}2^{m} & -a \\
   a     & 0
\end{pmatrix}\right)\begin{pmatrix}
   -1   & -a2^{n-1} \\
   a2^{n-1} & -1
\end{pmatrix}} \\
 &=& \overline{(-1)^{2^{n-m-1}}\begin{pmatrix}
   -1   & -a2^{n-1} \\
   a2^{n-1} & -1
\end{pmatrix}+(-1)^{2^{n-m-1}-1}(2^{n-1})\begin{pmatrix}
   -a^{2}(2^{m}+2^{n-1})   & -a^{3}2^{n+m-1}+a \\
   -a & -a^{2}2^{n-1}
\end{pmatrix}}\\
 &=& \overline{(-1)^{2^{n-m-1}}}\left(\begin{pmatrix}
   \overline{-1}   & \overline{-a2^{n-1}} \\
   \overline{a2^{n-1}} & \overline{-1}
\end{pmatrix}-\begin{pmatrix}
   \overline{0}   & \overline{a2^{n-1}} \\
   \overline{-a2^{n-1}} & \overline{0}
\end{pmatrix}\right) \\
 &=& \begin{pmatrix}
   \overline{-1}   & \overline{-a2^{n}} \\
   \overline{a2^{n}} & \overline{-1}
\end{pmatrix}~(n-m-1 \geq 1) \\
 &=& -Id.
\end{eqnarray*}

\end{proof}

\noindent On peut maintenant démontrer le résultat voulu.

\begin{proposition}
\label{314}

Soient $N=2^{n}$ avec $n \geq 4$, $2 \leq m \leq n-2$ et $a$ un entier impair. La solution $\overline{a2^{m}}$-monomiale minimale de \eqref{p} est réductible de taille $2^{n-m+1}$.

\end{proposition}

\begin{proof}

$(\overline{a2^{m}},\ldots,\overline{a2^{m}}) \in (\mathbb{Z}/N\mathbb{Z})^{2^{n-m+1}}$ est solution de \eqref{p} (voir lemme \ref{39}).
\\
\\On a \[(\overline{a2^{m}},\ldots,\overline{a2^{m}})=(\overline{a2^{m}-a2^{n-1}},\overline{a2^{m}},\ldots,\overline{a2^{m}},\overline{a2^{m}-a2^{n-1}}) \oplus (\overline{a2^{n-1}},\overline{a2^{m}},\ldots,\overline{a2^{m}},\overline{a2^{n-1}}),\] avec \[(\overline{a2^{n-1}},\overline{a2^{m}},\ldots,\overline{a2^{m}},\overline{a2^{n-1}}) \in (\mathbb{Z}/N\mathbb{Z})^{2^{n-m}+2}\] et \[(\overline{a2^{m}-a2^{n-1}},\overline{a2^{m}},\ldots,\overline{a2^{m}},\overline{a2^{m}-a2^{n-1}}) \in (\mathbb{Z}/N\mathbb{Z})^{2^{n-m}}.\] 
\\
\\De plus, par le lemme précédent, $(\overline{a2^{n-1}},\overline{a2^{m}},\ldots,\overline{a2^{m}},\overline{a2^{n-1}}) \in (\mathbb{Z}/N\mathbb{Z})^{2^{n-m}+2}$ est une solution de \eqref{p} de taille supérieure à 3. $(\overline{a2^{m}-a2^{n-1}},\overline{a2^{m}},\ldots,\overline{a2^{m}},\overline{a2^{m}-a2^{n-1}}) \in (\mathbb{Z}/N\mathbb{Z})^{2^{n-m}}$ est de taille supérieure à 3 (car $m \leq n-2$). 
\\
\\Donc, $(\overline{a2^{m}},\ldots,\overline{a2^{m}}) \in (\mathbb{Z}/N\mathbb{Z})^{2^{n-m+1}}$ est une solution réductible de \eqref{p}. Il nous reste à montrer que celle-ci est minimale.
\\
\\Notons $h$ la taille de la solution $\overline{a2^{m}}$-monomiale minimale de \eqref{p}. $h$ divise $2^{n-m+1}$. 
\\
\\De plus, $(\overline{a2^{m}-a2^{n-1}},\overline{a2^{m}},\ldots,\overline{a2^{m}},\overline{a2^{m}-a2^{n-1}}) \in (\mathbb{Z}/N\mathbb{Z})^{2^{n-m}}$ est une solution de \eqref{p}. Puisque $a$ est impair, $\overline{a2^{m}-a2^{n-1}} \neq \overline{a2^{m}}$ . Ainsi, par le lemme \ref{33} i), $h$ ne divise pas $2^{n-m}$. Donc, $h=2^{n-m+1}$.
\\
\\Ainsi, la solution $\overline{a2^{m}}$-monomiale minimale de \eqref{p} est réductible de taille $2^{n-m+1}$.

\end{proof}

Il nous reste à considérer les cas des solutions $\overline{2a}$-monomiales minimales de $(E_{2^{n}})$ lorsque $n \geq 2$. Comme précédemment, on commence par quelques résultats préliminaires :

\begin{lemma}
\label{315}

Soient $n \in \mathbb{N}^{*}$ et $k \in [\![1;n]\!]$, $\frac{n}{{\rm pgcd}(n,k)}$ divise $\binom{n}{k}$.

\end{lemma}

\begin{proof}

\[{n \choose k}=\frac{n}{k}{n-1 \choose k-1}=\frac{\frac{n}{{\rm pgcd}(n,k)}}{\frac{k}{{\rm pgcd}(n,k)}}{n-1 \choose k-1}.\] Donc, comme ${n \choose k} \in \mathbb{N}^{*}$, on a $\frac{k}{{\rm pgcd}(n,k)}$ divise $\frac{n}{{\rm pgcd}(n,k)}{n-1 \choose k-1}$. Comme $\frac{k}{{\rm pgcd}(n,k)}$ et $\frac{n}{{\rm pgcd}(n,k)}$ sont premiers entre eux, on a, par le lemme de Gauss, $\frac{k}{{\rm pgcd}(n,k)}$ divise ${n-1 \choose k-1}$. Donc, $\frac{n}{{\rm pgcd}(n,k)}$ divise ${n \choose k}$.

\end{proof}

\begin{lemma}
\label{316}

Soient $(n, j) \in \mathbb{N}^{2}$, $n \geq 3$ et $3 \leq j \leq n$. $2^{n+1-j}$ divise $\binom{2^{n-1}}{j}$.

\end{lemma}

\begin{proof}

Par le lemme \ref{315}, $\frac{2^{n-1}}{{\rm pgcd}(2^{n-1},j)}$ divise $\binom{2^{n-1}}{j}$. Il existe $u \in \mathbb{N}$ tel que ${\rm pgcd}(2^{n-1},j)=2^{u}$. Si $j=3$ alors $j$ est impair et $u=0$. En particulier, $2^{n-1}$ divise $\binom{2^{n-1}}{j}$ et donc $2^{n-2}$ également. On suppose donc, à partir de maintenant, $n \geq 4$ et $j \geq 4$.
\\
\\Montrons que si $j \geq 4$ on a $2^{j-2} \geq j$. On procède par récurrence sur $j$. $2^{4-2}=2^{2}=4$. Supposons qu'il existe un $j \geq 4$ tel que $2^{j-2} \geq j$. $2^{j+1-2}=2 \times 2^{j-2} \geq 2j \geq j+1$, ce qui prouve le résultat par récurrence.
\\
\\Ainsi, $u \leq j-2$ et donc $2^{n+1-j}$ divise $\frac{2^{n-1}}{{\rm pgcd}(2^{n-1},j)}$. On en déduit que $2^{n+1-j}$  divise $\binom{2^{n-1}}{j}$.

\end{proof}

\begin{lemma}
\label{317}

Soient $n \geq 3$, $N=2^{n+1}$ et $a$ un entier impair. $M_{2^{n}}(\overline{2a},\ldots,\overline{2a})=\begin{pmatrix}
   \overline{1+2^{n}a^{2}} & \overline{2^{n}a} \\
   \overline{-2^{n}a} & \overline{1+2^{n}a^{2}}
\end{pmatrix}$.

\end{lemma}

\begin{proof}

En procédant comme dans le lemme \ref{39}, on a en posant $D=\begin{pmatrix}
   2a^{2} & -a \\
   a     & 0
\end{pmatrix}$ et en utilisant $\overline{2^{n+1}}=\overline{0}$ :

\begin{eqnarray*}
M_{2^{n}}(\overline{2a},\ldots,\overline{2a}) &=& \overline{(-Id+2 D)^{2^{n-1}}} \\
																				      &=& \overline{ \sum_{j=0}^{2^{n-1}} \binom{2^{n-1}}{j} (-1)^{-j+2^{n-1}}2^{j}D^{j}}~{\rm (bin\hat{o}me~de~Newton)} \\
                                              &=& \overline{(-1)^{2^{n-1}}Id+(-1)^{-1+2^{n-1}}(2^{n-1})2D+(-1)^{-2+2^{n-1}}\binom{2^{n-1}}{2}4D^{2}} \\
                                              &+& \overline{\sum_{j=3}^{n} \binom{2^{n-1}}{j} (-1)^{-j+2^{n-1}}2^{j}D^{j}+ \sum_{j=n+1}^{2^{n-1}} \binom{2^{n-1}}{j} (-1)^{-j+2^{n-1}}2^{j}D^{j}} \\
                                              & & (2^{n-1}\geq n+1,~{\rm lemme}~\ref{312}~iii)) \\
																							&=& \overline{Id-2^{n}D+(2^{n}(2^{n-1}-1))D^{2}}~(2^{n+1-j}\mid \binom{2^{n-1}}{j},~{\rm lemme}~\ref{316}) \\
																							&=& \overline{Id-2^{n}D-2^{n}D^{2}}~(2n-1 \geq n+1) \\
																							&=& \overline{\begin{pmatrix}
   1 & 0 \\
   0 & 1
\end{pmatrix}-\begin{pmatrix}
   0 & -2^{n}a \\
   2^{n}a     & 0
\end{pmatrix}-\begin{pmatrix}
   -2^{n}a^{2} & 0 \\
    0     & -2^{n}a^{2}
\end{pmatrix}} \\
                                              &=& \begin{pmatrix}
   \overline{1+2^{n}a^{2}} & \overline{2^{n}a} \\
   \overline{-2^{n}a} & \overline{1+2^{n}a^{2}}
\end{pmatrix}. \\
\end{eqnarray*}

\end{proof}

\begin{proposition}
\label{318}

Soient $n \geq 2$, $N=2^{n}$ et $a$ un entier impair. La solution $\overline{2a}$-monomiale minimale de \eqref{p} est de taille $2^{n}$.

\end{proposition}

\begin{proof}

On raisonne par récurrence sur $n$.
\\
\\Si $n=2$ ou $n=3$ alors le résultat est vrai (voir Théorème \ref{32}).
\\
\\Supposons qu'il existe un $n \geq 3$ tel que la solution $\overline{2a}$-monomiale minimale de $(E_{2^{n}})$ est de taille $2^{n}$. Par le lemme \ref{39}, la taille $h$ de la solution $\overline{2a}$-monomiale minimale de $(E_{2^{n+1}})$ divise $2^{n+1}$. De plus, comme $2^{n}$ divise $2^{n+1}$, le $h$-uplet ne contenant que $2a$ est aussi une solution modulo $2^{n}$. Par hypothèse de récurrence, on a $2^{n}$ divise $h$. Ainsi, $h=2^{n}$ ou $h=2^{n+1}$. Or, par le lemme précédent, \[M_{2^{n}}(\overline{2a},\ldots,\overline{2a})=\begin{pmatrix}
   \overline{1+2^{n}a^{2}} & \overline{2^{n}a} \\
   \overline{-2^{n}a} & \overline{1+2^{n}a^{2}}
\end{pmatrix} \neq \pm Id~({\rm a~impair}).\]
\noindent Donc, $h \neq 2^{n}$ et $h=2^{n+1}$. Le résultat est démontré par récurrence.

\end{proof}

\noindent En combinant tous les résultats obtenus sur les tailles, on a le résultat suivant :

\begin{corollary}
\label{spe2}

Si $N=2^{n}$ avec $n \geq 2$, $1 \leq m \leq n$ et $a$ un entier impair alors la solution $\overline{a2^{m}}$-monomiale minimale de \eqref{p} est de taille $2^{n-m+1}$.

\end{corollary}

\begin{proof}

Soit $l$ la taille de la solution $\overline{a2^{m}}$-monomiale minimale de \eqref{p}. Si $m=1$ alors, par la proposition précédente, $l=2^{n}=2^{n-1+1}$. Si $2 \leq m \leq n-2$ alors, par la proposition \ref{314}, $l=2^{n-m+1}$. Si $m=n-1$ alors $\overline{a2^{m}}=\overline{\frac{N}{2}}$ et $l=4=2^{n-(n-1)+1}$, par le théorème \ref{32}. Si $m=n$ alors $\overline{a2^{m}}=\overline{0}$ et $l=2=2^{n-n+1}$.

\end{proof}

\begin{proposition}
\label{319}

Soient $n \geq 2$, $N=2^{n}$ et $a$ un entier impair. La solution $\overline{2a}$-monomiale minimale de \eqref{p} est irréductible.

\end{proposition}

\begin{proof}

On raisonne par récurrence sur $n$.
\\
\\Si $n=2$ ou $n=3$ alors le résultat est vrai (voir Théorème \ref{32}).
\\
\\Supposons qu'il existe un $n \geq 3$ tel que la solution $\overline{2a}$-monomiale minimale de $(E_{2^{n}})$ est irréductible. Supposons par l'absurde que la solution $\overline{2a}$-monomiale minimale de $(E_{2^{n+1}})$ est réductible. Il existe deux solutions de $(E_{2^{n+1}})$ de la forme $(\overline{x},\overline{2a},\ldots,\overline{2a},\overline{x})$ et de taille $h_{1}$ et $h_{2}$ avec $h_{1}, h_{2} \leq 2^{n+1}-1$ et $h_{1}+h_{2}=2^{n+1}+2$ (par la proposition précédente). De plus, comme $2^{n}$ divise $2^{n+1}$, ce $h_{1}$ (resp. $h_{2}$)-uplet est aussi une solution modulo $2^{n}$. 
\\
\\Par hypothèse de récurrence, la solution $\overline{2a}$-monomiale minimale de $(E_{2^{n}})$ est irréductible. Ainsi, on a $h_{1}, h_{2} \geq 2^{n}$ et donc on a trois possibilités :

\begin{itemize}
\item $h_{1}=2^{n}$ et $h_{2}=2^{n}+2$;
\item $h_{1}=2^{n}+1$ et $h_{2}=2^{n}+1$, impossible par le lemme \ref{33} ii);
\item $h_{1}=2^{n}+2$ et $h_{2}=2^{n}$.
\end{itemize}

\noindent Donc, il existe une solution de $(E_{2^{n+1}})$ de la forme $(\overline{x},\overline{2a},\ldots,\overline{2a},\overline{x})$ de taille $2^{n}+2$. Puisque ce $2^{n}+2$-uplet est une solution modulo $2^{n}$, on a, par le lemme \ref{33} iii), $2^{n}$ divise $x$. Donc, $\overline{x}=\overline{2^{n}}$ ou $\overline{x}=\overline{0}$.
\\
\\Supposons par l'absurde $\overline{x}=\overline{0}$. Comme $(\overline{0},\overline{2a},\ldots,\overline{2a},\overline{0})\sim (\overline{2a},\ldots,\overline{2a}) \oplus (\overline{0},\overline{0},\overline{0},\overline{0})$, $(\overline{2a},\ldots,\overline{2a}) \in (\mathbb{Z}/2^{n+1}\mathbb{Z})^{2^{n}}$ est une solution de $(E_{2^{n+1}})$. Ceci est absurde, par la proposition \ref{318}. Ainsi, $\overline{x}=\overline{2^{n}}$.
\\
\\Or,
 \begin{eqnarray*}
M_{2^{n}+2}(\overline{2a},\ldots,\overline{2a},\overline{2^{n}},\overline{2^{n}}) &=& M_{2}(\overline{2^{n}},\overline{2^{n}})M_{2^{n}}(\overline{2a},\ldots,\overline{2a}) \\
                                                                                  &=& \begin{pmatrix}
   \overline{-1} & \overline{-2^{n}} \\
   \overline{2^{n}} & \overline{-1}
\end{pmatrix}\begin{pmatrix}
   \overline{1+2^{n}a^{2}} & \overline{2^{n}a} \\
   \overline{-2^{n}a} & \overline{1+2^{n}a^{2}}
\end{pmatrix},~{\rm lemme}~\ref{317}  \\
                                                                                  &=& \begin{pmatrix}
   \overline{-1-2^{n}a^{2}} & \overline{-2^{n}(a+1)} \\
   \overline{2^{n}(1+a)} & \overline{-2^{n}a^{2}-1}
\end{pmatrix} \\
                                                                                  &=& \overline{(-2^{n}a^{2}-1)} Id~({\rm car}~a+1~{\rm pair}) \\
                                                                                  & \neq & \pm Id~({\rm car}~a~{\rm impair}). \\
\end{eqnarray*}

Ceci est absurde. Ainsi, la solution $\overline{2a}$-monomiale minimale de $(E_{2^{n+1}})$ est irréductible.
Le résultat est démontré par récurrence.

\end{proof}

On peut maintenant rassembler tous les éléments pour achever la démonstration de notre résultat principal.

\begin{proof}[démonstration du théorème \ref{26} ii)]

Soient $N=2^{n}$ avec $n \geq 1$ et $k \in [\![0;2^{n}-1]\!]$. Si $n \in \{1, 2, 3\}$ alors le résultat est vrai. On suppose donc $n \geq 4$. 
\\
\\Supposons $k$ impair. $\overline{k}$ est inversible, et, par la proposition \ref{35}, la solution  $\overline{k}$-monomiale minimale de \eqref{p} est irréductible.
\\
\\Supposons $k$ pair. La solution $\overline{0}$-monomiale minimale de \eqref{p} est réductible. On suppose $\overline{k} \neq \overline{0}$. Il existe un entier $1 \leq m \leq n-1$ et un entier $a$ impair tels que $k=a2^{m}$. On distingue plusieurs cas :
\begin{itemize}
\item si $m=n-1$ alors $a=1$ et, par le théorème \ref{32}, la solution  $\overline{k}$-monomiale minimale de \eqref{p} est irréductible;
\item si $m=1$ alors, par la proposition \ref{319}, la solution  $\overline{k}$-monomiale minimale de \eqref{p} est irréductible;
\item si $2 \leq m \leq n-2$ alors, par la proposition \ref{314}, la solution  $\overline{k}$-monomiale minimale de \eqref{p} est réductible.
\end{itemize}

\noindent Ceci conclut la classification. 
\\
\\Si $n=1$, \eqref{p} a une seule solution monomiale minimale irréductible. Si $n=2$, \eqref{p} a trois solutions monomiales minimales irréductibles. On suppose maintenant $n \geq 3$. Il y a $\varphi(2^{n})=2^{n}-2^{n-1}$ solutions $\overline{k}$-monomiales minimales de \eqref{p} avec $k$ impair. Il y a une solution $\overline{2^{n-1}}$-monomiale minimale de \eqref{p}. Il y a $2^{n-2}$ solutions $\overline{2a}$-monomiales minimales de \eqref{p} avec $a$ impair. En effet, pour avoir une telle solution, on cherche $a$ impair tel que $2a \leq 2^{n}$ c'est-à-dire $a$ impair inférieur à $2^{n-1}$. Comme pour $a$ impair, $\overline{2a} \neq \overline{2^{n-1}}$ (puisque $n \geq 3$), \eqref{p} a $2^{n}-2^{n-1}+2^{n-2}+1=3 \times 2^{n-2}+1$ solutions monomiales minimales irréductibles.

\end{proof}

\section{Application au calcul de la taille de certaines solutions monomiales minimales}
\label{taille}

Dans un précédent article, on avait établi que si $N=l^{n}$ alors le $2l^{n-1}$-uplet constitué uniquement de $\overline{l}$ est solution de \eqref{p} (voir \cite{M} proposition 3.14). Cependant, les questions de l'irréductibilité et de la minimalité restaient ouvertes et avaient été réunies dans un problème ouvert (voir \cite{M} problème 2). On avait ensuite démontré que cette solution est irréductible si et seulement si $l=2$ (voir \cite{M2} Théorème 3.9). On va maintenant résoudre entièrement ce problème en calculant la taille de certaines solutions monomiales minimales.  

\begin{proposition}
\label{41}

Soit $N=p_{1}^{\alpha_{1}} \ldots p_{r}^{\alpha_{r}}$ avec $r \geq 1$, $p_{i}$ des nombres premiers deux à deux distincts et $\alpha_{i}$ des entiers naturels non nuls. Soit $k=ap_{1}^{\beta_{1}} \ldots p_{r}^{\beta_{r}}$ avec, pour tout $i$, $1 \leq \beta_{i} \leq \alpha_{i}$ et $a$ un entier qui n'est divisible par aucun des $p_{i}$. La solution $\overline{k}$-monomiale minimale de \eqref{p} est de taille $2p_{1}^{\alpha_{1}-\beta_{1}} \ldots p_{r}^{\alpha_{r}-\beta_{r}}$.

\end{proposition}

\begin{proof}

Soit $h$ la taille de la solution $\overline{k}$-monomiale minimale de \eqref{p}. Si $r=1$ alors, par la proposition \ref{311} et le corollaire \ref{spe2}, le résultat est vrai. On suppose maintenant $r \geq 2$. On va distinguer deux cas :
\\
\\ \underline{Premier cas :} Si tous les $p_{i}$ sont impairs. 
\\
\\Soit $i \in [\![1;r]\!]$, posons $b_{i}=a \prod_{j \neq i} p_{j}^{\beta_{j}}$. On a $k+p_{i}^{\alpha_{i}} \mathbb{Z}=b_{i}p_{i}^{\beta_{i}}+p_{i}^{\alpha_{i}} \mathbb{Z}$ et $b_{i}$ n'est pas divisible par $p_{i}$. Par la proposition \ref{311}, la	solution $(k+p_{i}^{\alpha_{i}} \mathbb{Z})$-monomiale minimale de $(E_{p_{i}^{\alpha_{i}}})$, est de taille $2p_{i}^{\alpha_{i}-\beta_{i}}$ avec \[M_{2p_{i}^{\alpha_{i}-\beta_{i}}}(k+p_{i}^{\alpha_{i}} \mathbb{Z},\ldots,k+p_{i}^{\alpha_{i}} \mathbb{Z})=((-1)^{p_{i}^{\alpha_{i}-\beta_{i}}}+p_{i}^{\alpha_{i}} \mathbb{Z}) Id= (-1+p_{i}^{\alpha_{i}} \mathbb{Z}) Id.\]

\noindent Comme $p_{i}^{\alpha_{i}}$ divise $N$, $h$ est un multiple de $2p_{i}^{\alpha_{i}-\beta_{i}}$. Ainsi, $h$ est un multiple de \[{\rm ppcm}(2p_{i}^{\alpha_{i}-\beta_{i}}, i \in [\![1;r]\!])=2p_{1}^{\alpha_{1}-\beta_{1}} \ldots p_{r}^{\alpha_{r}-\beta_{r}}.\]

\noindent De plus, pour tout $i \in [\![1;r]\!]$, on a 

\begin{eqnarray*}
M_{2p_{1}^{\alpha_{1}-\beta_{1}} \ldots p_{r}^{\alpha_{r}-\beta_{r}}}(k+p_{i}^{\alpha_{i}} \mathbb{Z},\ldots,k+p_{i}^{\alpha_{i}} \mathbb{Z}) &=& [M_{2p_{i}^{\alpha_{i}-\beta_{i}}}(k+p_{i}^{\alpha_{i}} \mathbb{Z},\ldots,k+p_{i}^{\alpha_{i}} \mathbb{Z})]^{\prod_{j \neq i} p_{j}^{\alpha_{j}-\beta_{j}}}  \\
                              &=& ((-1+p_{i}^{\alpha_{i}} \mathbb{Z}) Id)^{\prod_{j \neq i} p_{j}^{\alpha_{j}-\beta_{j}}} \\
															&=& (-1+p_{i}^{\alpha_{i}} \mathbb{Z})Id. \\
\end{eqnarray*}		

\noindent Par le lemme chinois, $M_{2p_{1}^{\alpha_{1}-\beta_{1}} \ldots p_{r}^{\alpha_{r}-\beta_{r}}}(\overline{k},\ldots,\overline{k})=-Id$. Donc, $h=2p_{1}^{\alpha_{1}-\beta_{1}} \ldots p_{r}^{\alpha_{r}-\beta_{r}}$.
\\
\\ \underline{Deuxième cas :} Si un des $p_{i}$ est pair. Quitte à modifier l'ordre des facteurs, on suppose $p_{1}=2$.
\\
\\Soit $i \in [\![2;r]\!]$. Comme dans le cas précédent, la	solution $(k+p_{i}^{\alpha_{i}} \mathbb{Z})$-monomiale minimale de $(E_{p_{i}^{\alpha_{i}}})$, est de taille $2p_{i}^{\alpha_{i}-\beta_{i}}$ avec $M_{2p_{i}^{\alpha_{i}-\beta_{i}}}(k+p_{i}^{\alpha_{i}} \mathbb{Z},\ldots,k+p_{i}^{\alpha_{i}} \mathbb{Z})=((-1)^{p_{i}^{\alpha_{i}-\beta_{i}}}+p_{i}^{\alpha_{i}} \mathbb{Z})Id=(-1+p_{i}^{\alpha_{i}} \mathbb{Z}) Id.$
\\
\\ \noindent Si $\alpha_{1}=\beta_{1}$ alors la solution $(k+2^{\alpha_{1}} \mathbb{Z})$-monomiale minimale de $(E_{2^{\alpha_{1}}})$ est de taille 2 avec \[M_{2}(k+2^{\alpha_{1}} \mathbb{Z},k+2^{\alpha_{1}} \mathbb{Z})=(-1+2^{\alpha_{1}} \mathbb{Z}) Id.\] 
\noindent On peut alors procéder comme dans le premier cas et on obtient $h=2p_{1}^{\alpha_{1}-\beta_{1}} \ldots p_{r}^{\alpha_{r}-\beta_{r}}$. On suppose donc $\alpha_{1}>\beta_{1}$.
\\
\\Par le corollaire \ref{spe2}, la	solution $(k+2^{\alpha_{1}} \mathbb{Z})$-monomiale minimale de $(E_{2^{\alpha_{1}}})$ est de taille $2^{\alpha_{1}-\beta_{1}+1}$. De plus, par le lemme \ref{39}, \[M_{2^{\alpha_{1}-\beta_{1}+1}}(k+2^{\alpha_{1}} \mathbb{Z},\ldots,k+2^{\alpha_{1}} \mathbb{Z})=((-1)^{2^{\alpha_{i}-\beta_{i}}}+2^{\alpha_{1}} \mathbb{Z})Id= (1+2^{\alpha_{1}} \mathbb{Z}) Id.\]

\noindent Comme $p_{i}^{\alpha_{i}}$ divise $N$, $h$ est un multiple de $2p_{i}^{\alpha_{i}-\beta_{i}}$. Ainsi, $h$ est un multiple de \[{\rm ppcm}(2p_{i}^{\alpha_{i}-\beta_{i}}, i \in [\![1;r]\!])=2p_{1}^{\alpha_{1}-\beta_{1}} \ldots p_{r}^{\alpha_{r}-\beta_{r}}.\]

\noindent De plus, on a $M_{2p_{1}^{\alpha_{1}-\beta_{1}} \ldots p_{r}^{\alpha_{r}-\beta_{r}}}(k+2^{\alpha_{1}} \mathbb{Z},\ldots,k+2^{\alpha_{1}} \mathbb{Z})=(1+2^{\alpha_{1}} \mathbb{Z}) Id$ et pour tout $i \in [\![2;r]\!]$

\begin{eqnarray*}
M_{2p_{1}^{\alpha_{1}-\beta_{1}} \ldots p_{r}^{\alpha_{r}-\beta_{r}}}(k+p_{i}^{\alpha_{i}} \mathbb{Z},\ldots,k+p_{i}^{\alpha_{i}} \mathbb{Z}) &=& [M_{2p_{i}^{\alpha_{i}-\beta_{i}}}(k+p_{i}^{\alpha_{i}} \mathbb{Z},\ldots,k+p_{i}^{\alpha_{i}} \mathbb{Z})]^{2^{\alpha_{1}-\beta_{1}}\prod_{j \neq 1, i} p_{j}^{\alpha_{j}-\beta_{j}}}  \\
                              &=& ((-1+p_{i}^{\alpha_{i}} \mathbb{Z})Id)^{2^{\alpha_{1}-\beta_{1}}\prod_{j \neq 1, i} p_{j}^{\alpha_{j}-\beta_{j}}} \\
															&=& (1+p_{i}^{\alpha_{i}} \mathbb{Z}) Id. \\
\end{eqnarray*}		

\noindent Donc, par le lemme chinois, $M_{2p_{1}^{\alpha_{1}-\beta_{1}} \ldots p_{r}^{\alpha_{r}-\beta_{r}}}(\overline{k},\ldots,\overline{k})=Id$. Donc, $h=2p_{1}^{\alpha_{1}-\beta_{1}} \ldots p_{r}^{\alpha_{r}-\beta_{r}}$.									

\end{proof}

\begin{remark}

{\rm Si l'un des $\beta_{i}$ est nul alors la formule n'est plus forcément vraie. En effet, considérons $N=30=2 \times 3 \times 5$ et $k=6=2 \times 3$. La solution $\overline{k}$-monomiale minimale de \eqref{p} est de taille $12 \neq 10=2 \times 2^{1-1} \times 3^{1-1} \times 5^{1-0}$.
}

\end{remark}

\noindent Ceci nous permet, en particulier, de terminer la résolution du problème évoqué au début de la section.

\begin{corollary}
\label{42}

Soient $N=l^{n}$ avec $l, n \geq 2$. La solution $\overline{l}$-monomiale minimale de \eqref{p} est de taille $2l^{n-1}$.

\end{corollary}

\begin{proof}

Soit $l=p_{1}^{\alpha_{1}} \ldots p_{r}^{\alpha_{r}}$ avec $r \geq 1$, $p_{i}$ des nombres premiers deux à deux distincts et $\alpha_{i}$ des entiers naturels non nuls. Par la proposition précédente, la taille de la solution $\overline{l}$-monomiale minimale de \eqref{p} est $2 p_{1}^{n\alpha_{1}-\alpha_{1}} \ldots p_{r}^{n\alpha_{r}-\alpha_{r}}=2l^{n-1}$.

\end{proof}

\end{document}